\documentclass[reqno,11pt]{amsart}
\usepackage{amsmath,amsthm,amsfonts,amssymb,latexsym,epsfig}
\usepackage{hyperref}
\usepackage{enumerate}
\usepackage{times}
\usepackage[T1]{fontenc}
\newtheorem{theorem}{Theorem}[section]
\newtheorem{lemma}{Lemma}[section]

\newtheorem{defn}{Definition}[section]

\numberwithin{equation}{section}

\theoremstyle{remark}
\newtheorem{remark}[theorem]{Remark}
\renewcommand{\l}{\left}
\renewcommand{\r}{\right}

\newcommand{\lb}{\left}
\newcommand{\rb}{\right}

\begin{document}
\title{On a new parameter involving Ramanujan's theta-functions}
\author[S. Chandankumar, H. S. Sumanth Bharadwaj and Vijay Yadav]
{S. Chandankumar, H. S. Sumanth Bharadwaj and Vijay Yadav}
\address[S. Chandankumar and H. S. Sumanth Bharadwaj]{Department of Mathematics and Statistics, M. S. Ramaiah University of Applied Sciences, Bangalore-560 058, Karnataka, India.} \email{chandan.s17@gmail.com, sumanthbharadwaj@gmail.com}
\address[Vijay Yadav]{Department of Mathematics, S. P. D. T. College, Andheri (E),
	Mumbai - 400059, Maharashtra, INDIA.} \email{vijaychottu@yahoo.com}

\begin{abstract}
Srinivasa Ramanujan recorded explicit evaluations of certain quotients of theta functions in his lost notebook. Motivated by the works of Ramanujan, Jinhee Yi systematically studied the analogues of explicit evaluation of quotients of theta functions by defining parameters. In this work, we define a new parameter involving theta-functions and establish some modular relations to explicitly evaluate the parameter. \end{abstract}

\maketitle
\section{Introduction}
Ramanujan's contributions to the theory of theta functions \cite{RN} were significant and far-reaching. He developed his own theory of theta functions, which helped him to find many new results and properties in particular cases. He also rediscovered several theorems found in Jacobi's fundamental theta functions and triple product identity, which has numerous applications in the field of theta functions.

Ramanujan's theta functions are generalizations of the Jacobi theta functions, and they capture their general properties. In particular, the Jacobi triple product takes on a particularly elegant form when written in terms of the Ramanujan theta functions. Ramanujan's theta function $f(a,b)$ is defined by
\begin{equation*}
	f(a,b):=\sum_{n=-\infty}^{\infty}a^{n(n+1)/2} b^{n(n-1)/2},\,\,\,|ab|<1.
\end{equation*}
For complex numbers, $a$, $q $ with $|q|<1$, let $(a;q)_\infty$ = $\displaystyle \prod_{n=0}^\infty(1-aq^{n})$.

Using Jacobi's fundamental factorization formula \cite[Entry 19, p. 35]{BB3} $f(a,b)$ can be expressed in product as
\begin{equation}\label{jti}
	f(a,b)=(-a;ab)_\infty (-b;ab)_\infty(ab;ab)_\infty.
\end{equation}
Following theta-functions  $\varphi$, $\psi$ and $f$ are classical:
\begin{eqnarray}
\varphi(q)&:=&f(q,q)=\sum_{n=-\infty}^{\infty}q^{n^2} =(-q;q^2)^2_\infty(q^2;q^2)_\infty,\label{t1}\\
 \psi(q)&:=&f(q,q^3) =\sum_{n=0}^{\infty}q^{n(n+1)/2} =\frac{(q^2;q^2)_\infty}{(q;q^2)_\infty},\label{t2}\\
 f(-q)&:=&f(-q,-q^2)=\sum_{n=-\infty}^{\infty}(-1)^n q^{n(3n-1)/2}=(q;q)_\infty,\label{t3}
\end{eqnarray}
where the product representation in each of the last equality of \eqref{t1}--\eqref{t3} follows from \eqref{jti}. The Gaussian hypergeometric function is defined by
$$_2F_1(a,b;c;z):=\sum_{n=0}^{\infty}\frac{\lb(a\rb)_n \lb(b\rb)_n}{\lb(c\rb)_n n!}z^n,\ \ \ 0\leq|z|<1,$$
where $a$, $b$, $c$ are complex numbers, $c\neq0,-1,-2,\ldots$, and $$(a)_0=1,\ \ (a)_n=a(a+1)\cdots(a+n-1)\ \ \textrm{for any positive integer } n.$$
Now we recall the notion of a modular equation as understood by Ramanujan. The complete elliptic integral of the  first kind of modulus $k$, $0<k<1$ is defined by
\begin{equation}\label{ee11}
K(k):=\int_0^{\frac{\pi}{2}}\frac{d\phi}{\sqrt{1-k^2\sin^2\phi}}.
\end{equation}
Set $K'=K(k')$, where $k'=\sqrt{1-k^2}$  is the so--called complementary modulus of $k$.
It is classical to set $q(k)=e^{-\pi K(k')/K(k)}$ so that $q$ is one-to-one and increases from 0 to 1. 
A fundamental result in the theory of elliptic functions \cite[Entry 6, p. 101]{BB3} is given by
\begin{equation}\label{ee11}
	\varphi^2(q)=\frac{2}{\pi}K(k)=\sum_{n=0}^{\infty}\frac{\left(\frac{1}{2}\right)^2_n}{\left(n!\right)^2}k^{2n}.
\end{equation}
Let $L_1$ and ${L'_1}$ denote the complete elliptic integral of the first kind associated with the moduli $l_1$ and $l'_1,$ respectively. Suppose that the following equality
\begin{equation}\label{ee12}
n_1\frac{K'}{K}=\frac{L'_1}{L_1},
\end{equation}
holds for some positive integer $n_1$. Then a modular equation of degree $n_1$ is a relation between the moduli $k$ and $\ell_1$ which is induced by \eqref{ee12}.  Following Ramanujan, set $\alpha=k^2$ and $\beta=\ell_1^2$. We say that $\beta$ is of degree $n_1$ over $\alpha$. The multiplier $m$ corresponding to the degree $n_{1}$ is defined by
\begin{equation}\label{a1}
m:=\frac{K}{L_1}=\frac{\varphi^2(q)}{\varphi^2(q^{n_1})}.
\end{equation}

By using the transformation formulae for theta-functions recorded by Ramanujan for theta-functions $f(-q)$, $f(q)$ and $\varphi(q)$, Jinhee Yi \cite{jy1} introduced parameters $r_{k,n}$ and $r'_{k,n}$ as follows:

\begin{equation}\label{rkn}
r_{k,n}=\frac{f(-q)}{k^{1/4}q^{(k-1)/24}f({-q^k})},\,\,\,\textrm{where}\,\,\,q=e^{-2 \pi\sqrt{n/k}},
\end{equation}
\begin{equation}\label{rkkn}
r'_{k,n}=\frac{f(q)}{k^{1/4}q^{(k-1)/24}f({q^k})},\,\,\,\textrm{where}\,\,\,q=e^{-\pi\sqrt{n/k}}.
\end{equation}
 Yi \cite{jy2}, introduced two parameters $h_{k,n}$ and $h'_{k,n}$ as follows:
\begin{equation}\label{hkn}
h_{k,n}:=\frac{\varphi(q)}{k^{1/4}\varphi(q^k)},\,\,\,\textrm{where}\,\,\,q=e^{- \pi\sqrt{n/k}},
\end{equation}
\begin{equation}\label{hkkn}
h'_{k,n}:=\frac{\varphi(-q)}{k^{1/4}\varphi(-q^k)},\,\,\,\textrm{where}\,\,\,q=e^{-2 \pi\sqrt{n/k}},
\end{equation}
 and systematically studied several properties of the parameters and also found plethora of explicit evaluations of $r_{k,n}$, $r'_{k,n}$, $h_{k,n}$ and $h'_{k,n}$ for different positive real values of $n$ and $k$. She also established several new values of $\varphi(e^{-n\pi}).$

Adiga et. al. \cite{CA3}, derived a new transformation formula for $\psi(-q)$. Using this transformation formula Baruah and Nipen Saikia \cite{ND2}, and  Yi et. al. \cite{jy5} defined the parameters $l_{k,n}$ and $l'_{k,n}$:
\begin{equation}\label{lkn}
l_{k,n}:=\frac{\psi(-q)}{k^{1/4}q^{(k-1)/8}\psi(-q^k )},\,\,\,\textrm{where}\,\,\,q=e^{- \pi\sqrt{n/k}},
\end{equation}
\begin{equation}\label{lkkn} l'_{k,n}:=\frac{\psi(q)}{k^{1/4}q^{(k-1)/8}\psi(q^k )},\,\,\,\textrm{where}\,\,\,q=e^{- \pi\sqrt{n/k}}. \end{equation} They established several evaluations of $l_{k,n}$ and $l'_{k,n}$. Saikia \cite{NS1} by using a transformation formula recorded by Ramanujan introduced the following parameter $A_{k,n}$ as
\begin{equation}\label{ak1}
A_{k,n}=\frac{\varphi(-q)}{2k^{1/4}q^{k/4}\psi(q^{2k})},\ \ q=e^{-\pi\sqrt{n/k}}.
\end{equation}
He studied several properties and established some general theorems for the explicit evaluations of $A_{k,n}$. Motivated by these works, we define a new parameter $A'_{k,n}$ as:
\begin{defn}For any positive rationals $n$ and $k$, we have
\begin{equation}\label{ak2}
A'_{k,n}=\frac{\varphi(q)}{2k^{1/4}q^{k/4}\psi(q^{2k})},\ \ q=e^{-\pi\sqrt{n/k}},
\end{equation}
\end{defn}
and establish several evaluations of $A'_{k,n}.$

This work is organized as follows. Some notations and background results are listed in Section \ref{21}. We derive new modular equations involving theta-functions $\varphi$ and $\psi$ in Section \ref{31}. Several new explicit evaluations of $A_{k,n}$ and $A'_{k,n}$ are established in Section \ref{41}. In Section \ref{61} we establish some new modular relations for a modular function of Level 16 developed by Dongxi Ye \cite{Dye} and also explicitly evaluate the function. Finally, in Section \ref{s51} we list some general formulas for explicit evaluations of $h'_{2,n}$.

\section{Preliminary results}\label{21}
In this section, we list a few theta-function identities involving theta-functions $\varphi$ and $\psi$ which are useful in deriving modular equations. 

\begin{lemma}We have \cite[Entry 25 (ii), (iii), (iv), (v) and (vii),  p. 40]{BB3}
\begin{eqnarray}
&&\varphi(q)-\varphi(-q)=4q\psi(q^8),\label{d8}\\
&&\varphi(q)\varphi(-q)=\varphi^2(-q^2),\label{d12}\\
&&\varphi(q)\psi(q^2)=\psi^2(q),\label{d11}\\
&&\varphi^2(q)-\varphi^2(-q)=8q\psi^2(q^4),\label{d9}\\
&&\varphi^4(q)-\varphi^4(-q)=16q\psi^4(q^2).\label{d10}
\end{eqnarray}
\end{lemma}
\begin{lemma}\cite[Entry 10 (i), (iii), p. 122]{BB3} We have
	\begin{equation}\label{d4}
		\varphi(-q)=\varphi(q)(1-\alpha)^{1/8},
	\end{equation}
	\begin{equation}\label{d41}
		\varphi(-q^2)=\varphi(q)(1-\alpha)^{1/4}.
	\end{equation}
\end{lemma}
\begin{lemma}\cite[Entry 11 (i), (iii), (iv) and (v),  p. 123]{BB3} We have
\begin{equation}\label{d51}
\psi (q)=\sqrt{\frac{{z}}{2}}\left(\frac{\alpha}{q}\right)^{1/8},
\end{equation}
\begin{equation}\label{d52}
\psi (q^2)={\frac{\sqrt{z}}{2}}\left(\frac{\alpha}{q}\right)^{1/4},
\end{equation}
\begin{equation}\label{d53}
\psi (q^4)=\frac{\sqrt z\{1-\sqrt{1-\alpha}\}^{1/2}}{2\sqrt{2q}},
\end{equation}
\begin{equation}\label{d5}
\psi (q^8)=\frac{\sqrt z\{1-(1-\alpha)^{1/4}\}}{4q},
\end{equation}
\end{lemma}
where $z:=\ _2F_1\l(\frac{1}{2},\frac{1}{2};1;x\r)$ and $y:=\pi\dfrac{_2F_1\l(\frac{1}{2},\frac{1}{2};1;1-x\r)}{_2F_1\l(\frac{1}{2},\frac{1}{2};1;x\r)}$.

\begin{lemma}We have
\begin{enumerate}
\item \cite[Eq. (24.21), p. 215]{BB3} If $\beta $ is of degree $2$ over $\alpha $, then
\begin{equation}\label{d2}
\beta =\left( \frac{1-\sqrt{1-\alpha }}{1+\sqrt{1-\alpha }}\right) ^{2}.
\end{equation}
\item \cite[Entry 5 (ii), p. 230]{BB3} If $\beta $ is of degree $3$ over $\alpha $, then
\begin{equation}\label{d3}
\left( \alpha \beta \right)^{1/4}+\left( (1-\alpha )(1-\beta)\right)^{1/4}=1.
\end{equation}
\item \cite[Entry 13(i), p. 280]{BB3} If $\beta $ is of degree $5$ over $\alpha $, then
\begin{equation}\label{d7}
\left( \alpha \beta \right)^{1/2}+\left((1-\alpha )(1-\beta)\right)^{1/2} +2\left(\alpha \beta (1-\alpha)(1-\beta)\right)^{1/6}=1.
\end{equation}
\item \cite[Entry 19, p. 314]{BB3} If $\beta $ is of degree $7$ over $\alpha $, then
\begin{equation}\label{d6}
\left( \alpha \beta \right)^{1/8}+\left(\alpha \beta (1-\alpha)(1-\beta)\right)^{1/8}=1.
\end{equation}
\end{enumerate}
\end{lemma}
\begin{lemma}\cite[p. 233, (5.2)]{BB3} If $m=z_1/z_3$ and $\beta$ has degree 3 over $\alpha$, then
\begin{align}\label{d121}
\beta=\frac{(m-1)^3(3+m)}{16m}.
\end{align}
\end{lemma}
\begin{lemma}\cite[Entry 10 (i), p. 122]{BB3} We have
\begin{equation}\label{d42}
\varphi(q)=\sqrt z.
\end{equation}
\end{lemma}

\noindent We end this section by listing a few values of $r_{2,n}$ found by Yi in her thesis \cite{jy1}.
\begin{lemma}\label{r2n}We have\\
$r_{2,2}=2^{1/8}$,\,\, 
$r_{2,4}=2^{1/8}(\sqrt2+1)^{1/8}$,\,\, 
$r_{2,3}=(1+\sqrt2)^{1/6}$,\,\, 
$r_{2,9}=(\sqrt3+\sqrt2)^{1/3}$,\\ 
$r_{2,8}=2^{3/16}(\sqrt2+1)^{1/4}$,\,\,
$r_{2,6}=2^{1/24}(\sqrt3+1)^{1/4}$,\,\,
$r_{2,5}=\sqrt{\dfrac{1+\sqrt5}{2}}$,\\ 
$r_{2,7}=\sqrt{\dfrac{\sqrt2+1+\sqrt{2\sqrt2+1}}{2}}.$
%
\end{lemma}

\section{Some new New modular equations} \label{31}
In this section, we establish few theta-function identities involving $\varphi$ and $\psi$ which play key role in establishing some general theorems for explicit evaluations of $A_{4,n}$, $A'_{4,n}$, and $h'_{2,n}$.
\begin{lemma}
If $P:=\dfrac{\varphi (-q)}{q\psi (q^{8})}$ and $R:=\dfrac{\varphi(q)}{q\psi (q^{8})}$, then
\begin{equation}\label{39}
R=P+4.
\end{equation}
\end{lemma}
\begin{proof}
Invoking \eqref{d8}, we arrive at  \eqref{39}.
\end{proof}
\begin{lemma}
If $P:=\dfrac{\varphi (-q)}{\sqrt q\psi (q^{4})}$ and $R:=\dfrac{\varphi(q)}{\sqrt q\psi (q^{4})}$, then
\begin{equation}\label{310}
R^2=P^2+8.
\end{equation}
\end{lemma}
\begin{proof}
Using \eqref{d9}, we arrive at \eqref{310}.
\end{proof}
\begin{lemma}
If $P:=\dfrac{\varphi (-q)}{ q^{1/4}\psi (q^{2})}$ and $R:=\dfrac{\varphi(q)}{ q^{1/4}\psi (q^{2})}$, then
\begin{equation}\label{311}
R^4=P^4+16.
\end{equation}
\end{lemma}
\begin{proof}
Using \eqref{d10}, we arrive at \eqref{311}.
\end{proof}

\begin{lemma}
If $P:=\dfrac{\varphi (-q)}{ q^{1/8}\psi (q)}$ and $R:=\dfrac{\varphi(q)}{ q^{1/8}\psi (q)}$, then
\begin{equation}\label{312}
R^4=P^4+\frac{16}{R^4}.
\end{equation}
\end{lemma}
\begin{proof}
Using \eqref{d10}, we arrive at \eqref{311}.
\end{proof}

\begin{lemma}
If $P:=\dfrac{\varphi (-q)}{ \varphi (-q^{2})}$ and $Q:=\dfrac{\varphi(-q)}{ q\psi (q^{8})}$, then
\begin{equation}\label{313}
(Q+4)P^2=Q.
\end{equation}
\end{lemma}
\begin{proof}
Equation \eqref{d12} can be rewritten as
\begin{align}\label{314}
    \frac{\varphi^2(-q)}{\varphi^2(-q^2)}=\frac{\varphi(-q)}{\varphi(q)}.
\end{align}
Using \eqref{d8} in the above equation, we arrive at \eqref{314}.
\end{proof}

\begin{lemma}
If $P:=\dfrac{\varphi (-q)}{ \varphi (-q^{2})}$ and $Q:=\dfrac{\varphi(-q)}{\sqrt{q}\psi (q^4)}$, then
\begin{equation}\label{315}
(Q^2+8)P^4=Q^2.
\end{equation}
\end{lemma}
\begin{proof}
Squaring equation \eqref{314}, we have
\begin{align}\label{316}
    \frac{\varphi^2(q)}{\varphi^2(-q)}=\varphi^4(-q^2).
\end{align}
Using \eqref{d9} in the above equation, we arrive at \eqref{315}.
\end{proof}

\begin{lemma}
If $P:=\dfrac{\varphi (-q)}{ \varphi (-q^{2})}$ and $Q:=\dfrac{\varphi(-q)}{\sqrt[4]{q}\psi (q^2)}$, then
\begin{equation}\label{317}
(Q^4+16)P^8=Q^4.
\end{equation}
\end{lemma}
\begin{proof}
Squaring equation \eqref{316}, we have
\begin{align}\label{3171}
    \frac{\varphi^4(q)}{\varphi^4(-q)}=\varphi^8(-q^2).
\end{align}
Using \eqref{d10} in the above equation, we arrive at \eqref{3171}.
\end{proof}

\begin{lemma}
If $P:=\dfrac{\varphi (-q)}{ \varphi (-q^{2})}$ and $Q:=\dfrac{\varphi(-q)}{\sqrt[8]{q}\psi (q)}$, then
\begin{equation}\label{318}
(Q^8+16P^8)P^8=Q^8.
\end{equation}
\end{lemma}
\begin{proof}
From \eqref{d4} and \eqref{d41}, we have
\begin{equation}\label{319}
\alpha=1-P^8.
\end{equation}
Using equations \eqref{d4} and \eqref{d51}, we get
\begin{equation}\label{320}
Q^8=16\frac{(1-\alpha)^2}{\alpha}.
\end{equation}
Substituting for $\alpha$ from \eqref{319} in the above equation, we arrive at \eqref{318}.
\end{proof}

\begin{theorem}
If $P:=\dfrac{\varphi (q)}{q^{3/4}\psi (q^{6})}$ and $Q:=\dfrac{\varphi (q)}{\varphi(q^{3})}$, then
\begin{equation}\label{ad32}
(Q^8-6Q^4+8Q^2-3)P^4=256Q^6.
\end{equation}
\end{theorem}
\begin{proof}
Using \eqref{d52} and \eqref{d42}, we have
\begin{align}\label{d345}
P=\frac{2\sqrt m}{\beta^{1/4}}\quad \textrm{and} \quad Q=\sqrt m.
\end{align}
Invoking \eqref{d121} the above equation can be written as
\begin{align}\label{d3451}
\frac{P^4(Q^2-1)^3(3+Q^2)}{16Q^2})=(2Q)^4.
\end{align}
On factorizing above equation, we arrive at \eqref{ad32} which completes the proof.
\end{proof}
\begin{remark}
By transcribing \eqref{ad32} using the definition of $A'_{3,n}$ and $h_{3,n}$ we can find evaluation of $A'_{3,n}$ for some positive rationals $n$ by using the values of $h_{3,n}.$
\end{remark}
\begin{theorem}
If $P:=\dfrac{\varphi (-q)}{q\psi (q^{8})}$ and $Q:=\dfrac{\varphi (-q^2)}{q^2\psi (q^{16})}$, then
\begin{equation}\label{32}
P+\frac{8}{Q}+\frac{2P}{Q}+4=\frac{Q}{P}.
\end{equation}
\end{theorem}
\begin{proof}
Transcribing $P$ and $Q$ by using \eqref{d4} and \eqref{d5}, we obtain \begin{equation}\begin{split}\label{133}
\beta=1-\left(\frac{Q}{Q+4}\right)^4\ \ \ \textrm{and}\ \ \ \sqrt{1-\alpha}=\left(\frac{P}{P+4}\right)^2.
\end{split}
\end{equation}
Employing \eqref{133} in the equation \eqref{d2}, we arrive at
\begin{equation}\begin{split}\label{34}
&(-Q^2+4QP+8P+QP^2+2P^2)(Q^2P^2+4Q^2P+4Q^2+16QP\\&+32P+4QP^2+8P^2)=0.
\end{split}
\end{equation}
It is observed that for $|q|<1$, the second factor $Q^2P^2+4Q^2P+4Q^2+16QP+32P+4QP^2+8P^2\neq 0$. Thus the first factor of \eqref{34}
\begin{align*}
4QP+8P+QP^2+2P^2-Q^2=0.
\end{align*}
Dividing the above equation by $PQ$ and then rearranging, we arrive at \eqref{32}.
\end{proof}
\noindent Throughout this section, we set 
\begin{align}\label{abn}
A_n:=\frac{\varphi(-q)\varphi(-q^n)}{q^{n+1}\psi(q^8)\psi(q^{8n})}\qquad \textrm{and} \qquad B_n:=\frac{\varphi(-q)\psi(q^{8n})}{q^{1-n}\varphi(-q^n)\psi(q^8)}.
\end{align}

\begin{theorem}We have 
\begin{equation}\label{35}\begin{split}
&B^2_3+\frac{1}{B^2_3}=12\left(B_3+\frac{1}{B_3}\right)+\left(A_3+\frac{8^2}{A_3}\right)\\&+6\left(\sqrt{A_3}+\frac{8}{\sqrt{A_3}}\right)\left(\sqrt{B_3}+\frac{1}{\sqrt{B_3}}\right)+30.
\end{split}\end{equation}
\end{theorem}
\begin{proof}
Let us begin the proof by setting 
$$P:=\dfrac{\varphi (-q)}{q\psi (q^{8})}\,\,\, \textrm{and}\,\,\, Q:=\dfrac{\varphi (-q^3)}{q^3\psi (q^{24})}.$$
Transcribing $P$ and $Q$ by using the \eqref{d4} and \eqref{d5}, we obtain \begin{equation}\begin{split}\label{33}
\beta=1-\left(\frac{Q}{Q+4}\right)^4\ \ \ \textrm{and}\ \ \ \alpha=1-\left(\frac{P}{P+4}\right)^4.
\end{split}
\end{equation}
Ramanujan's modular equations of degree three in \eqref{d3} can be written as
\begin{equation}\begin{split}\label{36}
&\alpha\beta=\left(1-\{(1-\alpha)(1-\beta)\}^{1/4}\right)^4.
\end{split}
\end{equation}
Again, invoking \eqref{d4} and \eqref{d5} in the above equality and set $A_3:=PQ$ and $B_3=P/Q$, we arrive at \eqref{35} to complete the proof.
\end{proof}


\begin{theorem}We have 
\begin{equation}\label{372}\begin{split}
& B^3_5+\frac{1}{B^3_5}-1620=70\left(B^2_5+\frac{1}{B^2_5}\right)+785\left(B_5+\frac{1}{B_5}\right)+\left(A_5^2+\frac{8^4}{A_5^2}\right)\\& +80\left(\sqrt{A_5}+\frac{8}{A_5}\right)\left[5\left(\sqrt{B_5}+\frac{1}{\sqrt{B_5}}\right)+\left(\sqrt{B^3_5}+\frac{1}{\sqrt{B^3_5}}\right)\right]\\& +20\left(A_5+\frac{8^2}{A_5}\right)\left[5+2\left(B_5+\frac{1}{B_5}\right)\right]+10\left(\sqrt{A_5^3}+\frac{8^3}{\sqrt{A_5^3}}\right)\left(\sqrt{B_5}+\frac{1}{B_5}\right).
\end{split}\end{equation}
\end{theorem}
\begin{proof}
The proof of the \eqref{372} is similar to the proof of the equation \eqref{33}, except that  we use Ramanujan's modular equations of degree five \eqref{d7}, hence we omit the proof.
\end{proof}

\begin{theorem}We have
\begin{equation}\label{373}\begin{split}
& B^4_7+\frac{1}{B^4_7}=280\left(B^3_7+\frac{1}{B^3_7}\right)+9772\left(B^2_7+\frac{1}{B^2_7}\right)+60424\left(B_7+\frac{1}{B_7}\right)\\&\left(A_7^3+\frac{8^6}{A_7^3}\right)+\left(A_7^2+\frac{8^4}{A_7^2}\right)\left[203+84\left(B_7+\frac{1}{B_7}\right)\right]+28\left(\sqrt{A_7}+\frac{8}{\sqrt A_7}\right)\\&\times \left[1030\left(\sqrt B_7+\frac{1}{\sqrt B_7}\right)+313\left(\sqrt{B^3_7}+\frac{1}{\sqrt{B^3_7}}\right)+21\left(\sqrt{B^5_7}+\frac{1}{\sqrt{B^5_7}}\right)\right]\\&+140\left(\sqrt{A_7^3}+\frac{8^3}{\sqrt {A_7^3}}\right)\left[9\left(\sqrt B_7+\frac{1}{\sqrt B_7}\right) +2\left(\sqrt{B^3_7}+\frac{1}{\sqrt{B^3_7}}\right)\right]\\&+\left(A_7+\frac{8^2}{A_7}\right)\left[8092+4340\left(B_7+\frac{1}{B_7}\right)+546\left(B^2_7+\frac{1}{B^2_7}\right)\right]\\&+14\left(\sqrt{A_7^5}+\frac{8^5}{A_7^5}\right)\left(\sqrt{B_7}+\frac{1}{B_7}\right)+106330.\end{split}\end{equation}
\end{theorem}
\begin{proof}
The proof of the \eqref{372} is similar to the proof of the equation \eqref{33}, except that  we use Ramanujan's modular equations of degree seven \eqref{d6}, hence we omit the proof.
\end{proof}

\section{Explicit evaluation of $A_{k,n}$ and $A'_{k,n}$}\label{41}
In this section, we establish some general theorems for explicit evaluation of $A_{k,n}$ and $A'_{k,n}$ by using the modular equations established in Section \ref{31}. 
\begin{lemma}For any positive rational $n$, we have 
\begin{align}\label{a1by2r2n}
A_{1/2,n}=r^3_{2,n}.
\end{align}
\end{lemma}
\begin{proof}
By Entry 24 (iii) \cite[p. 39]{BB3}, we have
\begin{align}
\frac{\varphi(q)}{\psi(q)}=\frac{f^3(-q)}{f^3(-q^2)}.
\end{align}
By using the definition of $A_{k,n}$ for $k=$1/2 and $r_{k,n}$ for $k=$2, we arrive at \eqref{a1by2r2n}. 
\end{proof}


\begin{lemma}Let $n$ and $k$ be any two positive rational such that $k>1$, we have 
\begin{align}\label{a1by2r2n2}
A'_{k/2,n}=\frac{r^2_{2,4n}r_{k,4n}r_{2,k^2n}}{r^2_{2,n}}.
\end{align}
\end{lemma}
\begin{proof}
By using the definition of theta functions $\varphi(q)$, $\psi(q)$ and $f(-q)$ , we have
\begin{align}\label{rt12}
\frac{\varphi(q)}{\psi(q^k)}=\frac{(q^2;q^2)^{5}_{\infty}(q^k;q^k)_\infty}{(q;q)^2_\infty(q^4;q^4)^2_\infty(q^{2k};q^{2k})^2_\infty}=\frac{f^2(-q^2)}{f^2(-q)}\frac{f^2(-q^2)}{f^2(-q^4)}\frac{f(-q^2)}{f(-q^{2k})}\frac{f(-q^k)}{f(-q^{2k})}.
\end{align}
By using the definitions of $A'_{k,n}$ and $r_{k,n}$ in \eqref{rt12}, we arrive at \eqref{a1by2r2n2}. 
\end{proof}
\begin{theorem}\label{gtrms}For any positive real number $n$, we have
	\begin{enumerate}[(i)]
		\item $A'_{4,n}=A_{4,n}+\sqrt 2,$
		\item $A'_{2,n}=\sqrt{A^2_{2,n}+\sqrt 2},$
		\item $A'_{1,n}=\sqrt[4]{A^4_{1,n}+1},$
		\item $4(A'_{1/2,n})^4-(A'_{1/2,n})^{-4}=4A_{1/2,n}^4.$
	\end{enumerate}
\end{theorem}
\begin{proof}[Proof]
	To prove Theorem \ref{gtrms} (i), we use the definition of $A_{k,n}$ and $A_{k,n}$ with $n=4$ and \eqref{39}. Similarly Theorem \ref{gtrms} (ii) follows from \eqref{310}, Theorem \ref{gtrms} (iii) follows from \eqref{311}, Theorem \ref{gtrms} (iv) follows from \eqref{312}.
\end{proof}

\noindent By using the above Theorem \ref{gtrms}, we can establish some new explicit evaluations of $A'_{4,n}$, $A'_{2,n}$ and $A'_{1/2,n}$ for different rationals $n$.
\begin{theorem}\label{a4n}
We have
\begin{align}
A_{4,2}&=1+\sqrt{1+\sqrt 2},\label{411}\\
A_{4,3}&=\frac{(1+\sqrt{2})(1+\sqrt{3})}{\sqrt 2},\label{42}\\
A_{4,4}&=\sqrt{8+6\sqrt 2}+\sqrt{6+4\sqrt 2},\label{43}\\
A_{4,7}&=\frac{(3+\sqrt 7)(4+\sqrt{14})}{2},\label{44}\\
A_{4,8}&=\sqrt{52+36\sqrt 2+(32+24\sqrt2)a}+\sqrt{56+40\sqrt 2+(36+26\sqrt2)a},\label{45}
\end{align}
\begin{align}
A_{4,9}&=\sqrt{94+66\sqrt 2+38\sqrt 6+54\sqrt 3}+\sqrt{93+66\sqrt 2+38\sqrt 6+54\sqrt 3}\label{46},\\A_{4,12}&=\sqrt{402+232\sqrt3+284\sqrt2+164\sqrt 6}+\sqrt{416+240\sqrt3 +294\sqrt2+170\sqrt6},
\end{align}
\begin{align}
\nonumber A_{4,28}&=\sqrt{2(64641+17276\sqrt {14}+24432\sqrt7+45708\sqrt2)}\\&+2\sqrt{(32384+22899\sqrt2+12240\sqrt 7+8655\sqrt {14})},
\end{align}
where $a=\sqrt{1+\sqrt 2}.$
\end{theorem}
\begin{proof}[Proof of \eqref{411}]Transcribing \eqref{32} by using the definition of $A_{4,n}$, we have   
\begin{equation}\label{4114}
A^2_{4,4n}=4A_{4,4n} A_{4,n}+2\sqrt 2 A_{4,n}+2\sqrt 2A_{4,4n}A^2_{4,n}+2A^2_{4,n}.
\end{equation}
Set $n=1/2$ in the above equation and using the fact that $A_{4,2}A_{4,1/2}=1$, we have
\begin{equation}\label{4112}
(A^2_{4,2}-2A_{4,2}-\sqrt 2)(A^2_{4,2}+2A_{4,2}+\sqrt2)=0.
\end{equation}
Since the second factor has not real roots, recalling that $A_{4,2}>1$ and solving the equation $A^2_{4,2}-2A_{4,2}-\sqrt 2=0$, we arrive at \eqref{411}.
\end{proof}

\begin{proof}[Proof of \eqref{42}]Transcribing \eqref{35} by using the definition of $A_{4,n}$, we have   
\begin{equation}\begin{split}
& 8{A^3_{4,n}}{A^3_{4,9n}}+12{\sqrt 2}{A^3_{4,n}}{A^2_{4,9n}}+12{A^3_{4,n}}{A_{4,9n}}+12{\sqrt 2}{A^3_{4,9n}}{A^2_{4,n}}\\&+30{A^2_{4,9n}}{A^2_{4,n}}+12{\sqrt 2}{A_{4,9n}}{A^2_{4,n}}+12{A^3_{4,9n}}{A_{4,n}}+12{\sqrt 2}{A^2_{4,9n}}A_{4,n}\\&+8{A_{4,n}}{A_{4,9n}}=A^4_{4,n}+{A^4_{4,9n}}.
\end{split}\end{equation}
Set $n=1/2$ in the above equation and using the fact that $A_{4,3}A_{4,1/3}=1$, we have
\begin{equation}\label{4113}\begin{split}
&(A^2_{4,3}-2A_{4,3}-{\sqrt 2}A_{4,3}-3-2{\sqrt 2})(A^2_{4,3}+{\sqrt 2}A_{4,3}+1)^2\\&(A^2_{4,3}+2A_{4,3}-{\sqrt 2}A_{4,3}-3+2{\sqrt 2})=0,\,\, \end{split}\end{equation}
Since the second factor has no real roots, recalling that $A_{4,3}>1$ and solving $A^2_{4,3}-(2+\sqrt 2)A_{4,3}-3-2\sqrt 2=0$, we arrive at \eqref{42}.
\end{proof}
\begin{proof}[Proof of \eqref{43}] Setting $n=1$ in \eqref{4114} and using the fact that $A_{4,1}=1$, we get    
\begin{equation}
A^2_{4,4}=4A_{4,4}+2{\sqrt 2}+2{\sqrt 2}A_{4,4}+2.
\end{equation}
Solving the above equation and recalling that $A_{4,4}>1$, we arrive at \eqref{43}.
\end{proof}

\begin{proof}[Proof of \eqref{44}] Transcribing $A_7$ and $B_7$, defined in \eqref{abn} along with the definition of $A_{4,n}$, we have
\begin{equation}
A_7=8A_{4,49n}A_{4,n}\,\,\textrm{and}\,\, B_7=A_{4,49n}/A_{4,49n}, 
\end{equation}
 Set $n=1/7$ in the above equation and using the fact that $A_{4,7}A_{4,1/7}=1$, then $A_7=8$ and $B_7=1/A^2_{4,7}$. Using $A_7$ and $B_7$ in  \eqref{373} and factorizing, we arrive at 
\begin{equation}\begin{split}
&(h^2-12h-7{\sqrt 2}h+1)(h^2+12h-7{\sqrt 2}h+1)(h^2+{\sqrt 2}h+1)^2\\&\times(h^2+2h+3{\sqrt 2}h+1)^2(h^2-2h+3{\sqrt 2}h+1)^2=0.
\end{split}\end{equation}
where $h=A_{4,7}$.

Recalling that $A_{4,7}>1$ and solving $h^2-(12+7\sqrt 2)h+1=0$, we get \eqref{44}.
\end{proof}
 
We observe that \eqref{4114} results in a quadratic equation  in $A_{4,4n}$ for any known value of ${A_{4,n}}$. We use the value of $A_{4,2}$ to get $A_{4,8}$, $A_{4,3}$ to get $A_{4,12}$, $A_{4,7}$ to get $A_{4,28}$ respectively. Hence we omit the proof.  The values of $A_{4,1/n}$ where $n \in \{2,3,4,7,8,9,12, 24\}$ can be easily found out by the fact that $A_{4,n}A_{4,1/n}=1.$ 
\section{Modular function of level 16}\label{61}
For $|q|<1$, Ye \cite{Dye} developed and studied a modular function of Level 16 which is an analogue of Ramanujan's theories of elliptic functions to alternate bases: 
\begin{equation}
h(q)=q\prod\limits_{j=1}^{\infty}\frac{(1-q^{16j})^2(1-q^{2j})}{(1-q^j)^2(1-q^{8j})}.
\end{equation}
In his work, he established some basic properties involving $h$ and Ramanujan's theta function. In the following Lemma, we list one of the relation. \begin{lemma}We have
\begin{equation}\label{hq}
    h(q)=\frac{q\psi(q^8)}{\varphi(-q)}.
\end{equation}
\end{lemma}
He established modular relation for $h(q)$ connecting with $h(-q)$, $h(q^2)$, $h(q^4)$ and $h(q^8).$
Note that in Section \ref{31}, we have provided algebraic relations between $\dfrac{\varphi (-q)}{q\psi (q^{8})}$ and $\dfrac{\varphi (-q^n)}{q^n\psi (q^{8n})}$ for $n \in 3$ and $5$ using which we can establish relation connecting $h(q)$ with $h(q^n)$, for $n \in  3$ and $5$. We list the relation in the following Theorem, set $u:=h(q)$ and $v_n:=h(q^n).$ 

\begin{theorem}\label{hq35}We have
\begin{enumerate}
\item [i)] $ u^4+(-12{v_3}-48{v_3}^2-64{v_3}^3)u^3+(-6{v_3}-48{v_3}^3-30{v_3}^2)u^2\\ +(-6{v_3}^2-12{v_3}^3-{v_3})u+{v_3}^4=0.$
\item [ii)] $(10u^2+80u^4+70u^5+u+40u^3)v_5+(4096u^5+640u^2+70u+5120u^4+2560u^3)v_5^5+(6400u^4+785u^2+80u+3200u^3+5120u^5)v_5^4+(3200u^4+400u^2+1620u^3
+40u+2560u^5)v_5^3+(100u^2+10u+785u^4+640u^5+400u^3)v_5^2=v_5^6+u^6.$
\end{enumerate}
\end{theorem}
\begin{proof}
Theorem \ref{hq35} follows from the definition of $h(q)$ and equations  \eqref{35} and \eqref{372}.
\end{proof}
\begin{theorem}
For any positive real number $n$, we have
\begin{equation}\label{exphq}
h(e^{-\pi\sqrt{\frac{n}{8}}})=\frac{1}{\sqrt{8}A_{4,n}}.
\end{equation}
\end{theorem}
\begin{proof}
By using the definition of $h(q)$ and $A_{4,n}$, we arrive at \eqref{exphq}.
\end{proof}
 \begin{lemma}We have
 \begin{align}
 & h(e^{-\frac{\pi}{2\sqrt 2}})=2^{-3/2},\label{q1}\\
 &h(e^{-\frac{\pi}{ 2}})=2^{-2}\left(\sqrt{1+\sqrt{2}}-1\right),\label{q2}\\
 & h(e^{-\frac{\pi}{ 4}})=2^{-3/2}\left(\sqrt{1+\sqrt{2}}+1\right),\label{q3}
\\&h(e^{-\pi\sqrt{\frac{3}{8}}})=2^{-2}(\sqrt 3-1)(\sqrt 2-1),\\&
h(e^{-\frac{\pi}{\sqrt {24}}})=2^{-2}(\sqrt 3+1)(\sqrt 2+1),\\
&
h(e^{-\frac{\pi}{\sqrt 2}})=\frac{\sqrt
{8+6\sqrt 2)}-\sqrt{6+4\sqrt 2}}{4\sqrt 2(\sqrt 2+1)},\\&
h(e^{-\frac{\pi}{4\sqrt 2}})=\frac{\sqrt
{8+6\sqrt 2)}+\sqrt{6+4\sqrt 2}}{\sqrt 8}.\label{q8}
\end{align}
 \end{lemma}
The evaluations listed in above Lemma follows easily by using the values of $A_{4,n}$ for $n=$ 1, 2, 1/2, 3, 1/3, 4 and 1/4 respectively in equation \eqref{exphq}. 
\section{Explicit evaluation of $h'_{2,n}$ }\label{s51}
In this section, we list applications of modular equations found in Section \ref{31} to derive few relations connecting the parameters $A_{k,n}$ with $h'_{2,n}$. We begin this section by listing a few explicit evaluations of $h'_{2,n}$ for different real number $n.$
\begin{theorem}\label{hn/8}
We have
\begin{align}
h'_{2,1/8}&=2^{-1/4}\sqrt{\sqrt 2-1},\label{h2by8}\\
h'_{2,3/8}&=2^{-1/4}\sqrt{(\sqrt2+1)(\sqrt3+\sqrt2)}\label{h2by38},\\
h'_{2,7/8}&=2^{-1/4}(\sqrt 2-1)\sqrt{2\sqrt 2+\sqrt 7},\label{h2by78}\\
h'_{2,9/8}&=2^{-1}{\left[(\sqrt3-\sqrt2)^2(4+(12-8\sqrt2)a)+12\sqrt3-14\sqrt2\right]^{1/2}},\label{h2by98}\\
h'_{2,1/16}&=2^{-1/2}\left({\sqrt{\sqrt{2}+1}-1}\right),\label{h2by16}\\
h'_{2,1/24}&=2^{-1/4}\sqrt{(\sqrt2-1)(\sqrt3-\sqrt2)},\label{h2by24}\\
h'_{2,1/32}&=2^{-1}{\sqrt{(8+6\sqrt2)-4\sqrt{8+6\sqrt2}}},\label{h2by32}\\
h'_{2,1/56}&=2^{-1/4}(\sqrt 2-1)\sqrt{2\sqrt 2-\sqrt 7},\label{h2by56}
\end{align}
where $a:=\sqrt{93+66\sqrt 2+38\sqrt 6+54\sqrt 3}.$
\end{theorem}
\begin{proof}
When we transcribe \eqref{313} by using the definition of $h'_{2,n}$ and $A_{4,n}$, we get
\begin{align}\label{ah4}
    \sqrt2A_{4,n}+2(h'_{2,n/8})^2=A_{4,n},
\end{align}
Note that the above equation is a general formula to explicitly evaluate $h'_{2,n/8}$ for any positive rationals $n$. We prove a few values listed in above Theorem \ref{hn/8}. For a proof of \eqref{h2by8}, set  $n=1$ in \eqref{ah4} and using the fact that $A_{4,1}=1$, we find that
\begin{equation}
2(h'_{2,1/8})^2+\sqrt 2=2.
\end{equation}
Solving the above equation for $h'_{2,1/8}$ and recalling that $h'_{2,1/8}>1$, we get \eqref{h2by8}.

\noindent For proving \eqref{h2by38}--\eqref{h2by56}, we repeat the same argument as in the proof of \eqref{h2by8}.
\end{proof}
\begin{theorem}
We have
\begin{align}
h'_{2,1/4}&=2^{-3/4}\sqrt[4]{\sqrt 2-1},\label{h1by4}\\
h'_{2,3/4}&=2^{-1}{\sqrt[4]{(16-8\sqrt3)(\sqrt3+\sqrt2)}}\label{h2by34},\end{align}\begin{align}
h'_{2,1/12}&=2^{-1}{\sqrt[4]{(16-8\sqrt3)(\sqrt3-\sqrt2)}}\label{h2by12},
\\
h'_{2,9/4}&=2^{-1}{\sqrt[4]{(\sqrt2-1)^3(56+32\sqrt3)}}\label{h2by19},\\
h'_{2,1/36}&=2^{-1}{\sqrt[4]{(\sqrt2-1)^3(56-32\sqrt3)}}\label{h2by36}.
\end{align}
\end{theorem}
\begin{proof}
When we transcribe \eqref{315} by using the definition of $h'_{2,n}$ and $A_{2,n}$, we get 
\begin{align}\label{ah2}
    2A^2_{2,n}(h'_{2,n/4})^4+2\sqrt 2(h'_{2,n/4})^4=A^2_{2,n},
\end{align}
 Note that the above equation is a general formula to explicitly evaluate $h'_{2,n/4}$ for any positive rationals $n$. For brevity we prove \eqref{h1by4} and \eqref{h2by34}. Put $n=1$ in \eqref{ah2} and using the fact that $A_{2,1}=1$, we have
\begin{align}\label{ah21}
    \sqrt 2=2(h'_{2,1/4})^4+1,
\end{align}
On solving the above equation and recalling that $h'_{2,1/4}>1,$ we arrive at \eqref{h1by4}.

For proof of \eqref{h2by34}, we let $n=3$ and using the value of $A_{2,3}=\sqrt{(\sqrt3 +\sqrt2)(\sqrt2+1)}$ found by Saikia \cite{NS1} in \eqref{ah2}, we get
\begin{align}\label{ah22}
    2(h'_{2,3/4})^4+\sqrt 6-2\sqrt 3-2\sqrt 2+3=0.
\end{align}
On solving \eqref{ah22} and recalling that $h'_{2,3/4}>1$, we arrive at \eqref{h2by34}.
\end{proof}
For \eqref{h2by12}-\eqref{h2by36}, we repeat the same argument as in the proof of \eqref{h2by34} except we use the values of $A_{2,1/3}$, $A_{2,9}$ and $A_{2,1/9}$ respectively obtained by Saikia \cite{NS1}.

\begin{theorem}If $h:=h'_{2,n/2}$\,\,\,and\,\,\,$A:=A_{1,n}$\,\,\,then
\begin{align}\label{ah1by2}
    4(A^4+1)h^8=A^4.
\end{align}
\end{theorem}
\begin{proof}
Transcribing \eqref{317} by using the definition of $h'_{2,n}$ and $A_{1,n}$, we arrive at \eqref{ah1by2}. 
\end{proof}

\begin{theorem}
\begin{eqnarray}
&&h'_{2,1}=2^{7/8}(\sqrt2-1)^{1/8},\label{h21}\\
&&h'_{2,3}= 2^{-1/8} (-17-12\sqrt2+10\sqrt3+7\sqrt6)^{1/8}\label{h23},
\\ 
&&h'_{2,5}=2^{-1/8}(-161-72\sqrt5+51\sqrt{10} +114\sqrt2)^{1/8},\label{h25}\\&&
h'_{2,6}=(85\sqrt6+147\sqrt2-208-120\sqrt3)^{1/8},\label{h26}\\&&
h'_{2,9}=2^{-1/8}(2772\sqrt3-4801+3395\sqrt2-1960\sqrt6)^{1/8}\label{h29}.
\end{eqnarray}
\end{theorem}
\begin{proof}
When we transcribe \eqref{318} by using the definition of $h'_{2,n}$ and $A_{1/2,n}$, we get 
\begin{align}\label{ah211}
   (4h^8-1)A^8+4h^{16}=0,
\end{align}
where $h:=h'_{2,n}$ and $A:=A_{1/2,n}$ Note that the above equation is general formula to explicitly evaluate $h'_{2,n}$ for any positive rationals $n$. For brevity we prove \eqref{h21}. Put $n=1$ in \eqref{ah211} and using the fact that $A_{1/2,1}=1$, we have
\begin{align}
(h'_{2,1})^8+(h'_{2,1})^{16}=2^{-2}.
\end{align}
On solving the above equation and recalling that  $h'_{2,1}>0,$ we arrive at \eqref{h21}.

For proving \eqref{h23}-\eqref{h29}, we use \eqref{ah211} along with the corresponding value of $A_{1/2,n}$ and repeat the same argument as in the proof of \eqref{h2by34}, to complete
the proof.
\end{proof}


\section{Concluding Remarks}
In this article, we have explicitly evaluated the parameter $h'_{2,n}$ for few positive rational $n$. We have explored a few properties of the parameters $A_{k,n}$ and $A'_{k,n}$ and established a few relations connecting each other. 
\section{Acknowledgments}
The authors are grateful to the anonymous referee for his/her valuable suggestions.

\end{document}